\documentclass[11pt]{amsart}

\usepackage{amssymb}
\usepackage{amsthm}
\usepackage{amsmath}
\usepackage{graphicx}
\usepackage{comment}
\usepackage[usenames,dvipsnames]{xcolor}
\usepackage[margin = 1in]{geometry}
\usepackage{enumerate}
\usepackage[toc,page]{appendix}
\usepackage{lineno}
\usepackage{color}
\usepackage{hyperref}

\definecolor{mygreen}{rgb}{0.1,0.75,0.2}

 \newtheorem{thm}{Theorem}[section]
 
 \newtheorem{lem}[thm]{Lemma}
 \newtheorem{prop}[thm]{Proposition}
  \newtheorem{asm}[thm]{Assumption}

 \theoremstyle{definition}

 \theoremstyle{remark}
 \newtheorem{rem}{Remark}

 \numberwithin{equation}{section}

\providecommand{\bbs}[1]{\left(#1\right)}

\newcommand{\bR}{\mathbb{R}}

\newcommand{\pt}{\partial}
\newcommand{\eps}{\varepsilon}

\newcommand{\ud}{\,\mathrm{d}}
\newcommand{\8}{\infty}

\begin{document}

\title[Asymptotic stability for dynamic boundary reaction]{Asymptotic stability for diffusion with dynamic boundary reaction from Ginzburg-Landau energy}

\author{Yuan Gao}
\address{Department of Mathematics, Purdue University, West Lafayette, IN,
  47907, USA}
\email{gao662@purdue.edu}

\author{Jean-Michel Roquejoffre}

\address{Institut de Math\'ematiques de Toulouse; UMR 5219 Universit\'e de Toulouse; CNRS
Universit\'e Toulouse III, 118 route de Narbonne, 31062 Toulouse, France}
\email{jean-michel.roquejoffre@math.univ-toulouse.fr}

\keywords{Long time behavior, metastability, algebraic decay, boundary stabilization, double well potential}

\subjclass[2010]{35K57, 35B35, 74H40}

\date{\today}

\begin{abstract}
The nonequilibrium process in dislocation dynamics and its relaxation to the metastable transition profile is crucial for understanding the plastic deformation caused by line defects in materials.  In this paper, we consider the full dynamics of a scalar dislocation model in two dimensions described by the bulk diffusion equation coupled with dynamic boundary condition on the interface, where a nonconvex misfit potential, due to the presence of dislocation, yields an interfacial reaction term on the interface. We prove the dynamic solution to this bulk-interface coupled system will uniformly converge to the metastable transition profile, which has a bi-states with fat-tail decay rate at the far fields. This global stability for the metastable pattern is the first result for a bulk-interface  coupled dynamics driven only by an interfacial reaction on the slip plane. 
\end{abstract}

\maketitle


\section{Introduction}
Metastable pattern formations are fundamentally important processes in materials science. The associated nonequilibrium dynamics is usually determined by the internal microscopic structure but can also be approximated by a macroscopic model after incorporating some nonlinear interfacial potentials.

In this paper, we study 
the  relaxation process to a metastable  transition profile for the following  full dynamics in terms of a scalar displacement function  $u(t, x,y)$  
\begin{equation}\label{maineq}
\left\{ \begin{array}{c}
\pt_t u - \Delta u=0, \quad y>0;\\
\pt_t u - \pt_y u + W'(u)=0, \quad y=0,
\end{array} \right. 
\end{equation}
where $W$ is a double well potential function with equal minima $W(\pm 1)$ and satisfies \eqref{potential} below. 
The main goal is to obtain the uniform convergence to a nontrivial steady solution, i.e., a metastable transition profile $\phi(x,y)$ connecting $\pm1$ at far fields $x\to \pm\8$; see Fig. \ref{fig1} (Right). Thus we assume for any fixed $y\geq 0$, the
 initial data $u_0(x,y)$ has bi-states $\pm 1$ as $x\to \pm\8$, which is specifically described in Assumption \ref{asm1}.

This model is motivated by nonlinear  dislocation dynamics, which consists of  the dynamics of the elastic continua for $y>0$ and the nonlinear reaction induced by the interfacial misfit potential $W$ on the slip plane $\Gamma:=\{(x,y)\in \bR^2;\,  y=0\}.$ The static dislocation model incorporating the atomistic misfit on the interface using $W$ was first proposed  by \textsc{Peierls and Nabarro} \cite{Peierls, Nabarro}  to study the atomic core structure near the dislocation line; see Fig. \ref{fig1}(left) and detailed physical derivations in Appendix \ref{app1}.
The presence of  dislocation is represented by a nonlinear interfacial
potential $W$ on the interface $\Gamma$.
Assume the double well/periodic potential $W$ satisfies
\begin{equation}\label{potential}
\begin{aligned}
 &W \in C_b^{3}(\mathbb{R}; \bR), \vspace{1ex} \\ &W(x)>W(1)=W(-1),\quad x \in \left(-1, 1 \right), \vspace{1ex} \\
 &  W''\left(\pm 1\right)>0, \quad W'(\pm 1)=0.
 \end{aligned}
 \end{equation}
For presentation's simplicity, we also assume $W'(0)=0$. Notice here the unstable state $0$ and the stable states $\pm 1$ can be chosen as other  generic constants without loss of generality.

The motion of dislocations, the most common line defects in materials science,  will lead to plastic deformations. 
Unlike previous related results in mathematical analysis, which assume  quasi-static elastic bulks, i.e., $\Delta u=0$  or assume a static Lam\'e system for $y>0$, 
 the full dynamics  of dislocations in \eqref{maineq} for the elastic bulks $y>0$ and the slip plane $\Gamma$ are coupled together. Physically, the fully coupled system exchanges both the mass and the energy on the bulk-boundary interface. This interaction between bulk and interface is very common in materials while the global stability analysis for the metastable equilibrium  of this kind of bulk-interface interactive dynamics is absent in the literature.  The global stability result in this paper will unveil the relaxation process of materials with dislocation structure. Precisely, starting from a perturbed initial data, probably due to impulsive stress, the full dynamics of the materials will eventually converge to metastable steady profile.
Notice there is no Dirichlet to Neumann map $\pt_n u  =(-\Delta)^{\frac12}{u} |_\Gamma$  to reduce \eqref{maineq} to a 1D nonlocal diffusion-reaction equation (see \eqref{quasi}) only on the interface $\Gamma$. 
Therefore in the interactive dynamics \eqref{maineq}, whether all the dynamic solutions of \eqref{maineq} will uniformly converge to a single metastable transition profile (see $\phi(x,y)$ below)  and the uniform relaxation rate  are still open.
\begin{figure}
\includegraphics[scale=0.34]{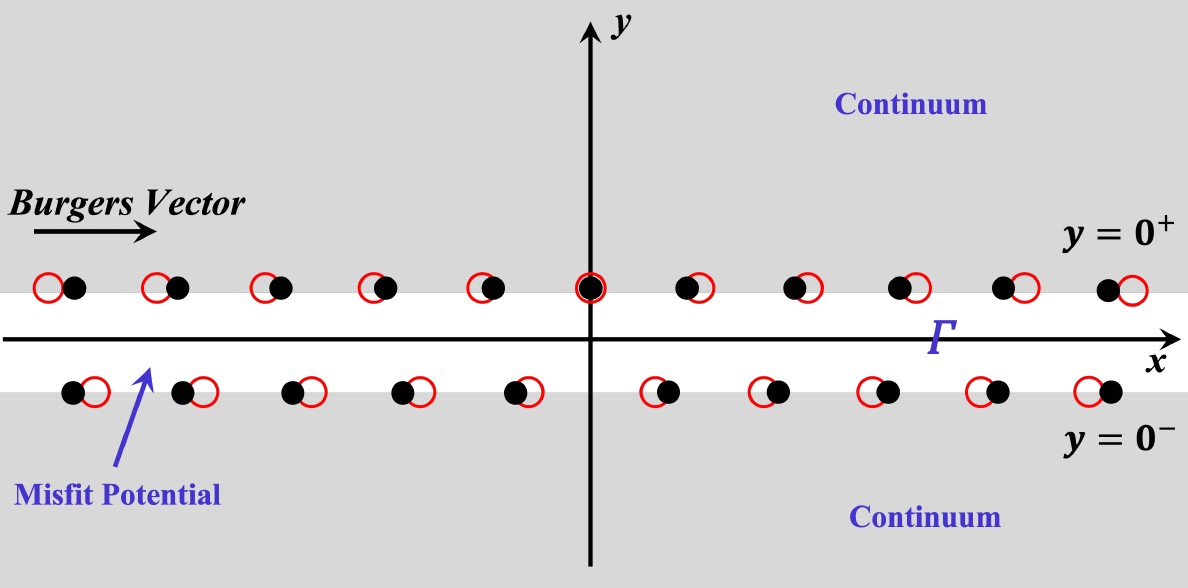}
\hspace{0.2in}
 \includegraphics[scale=0.38]{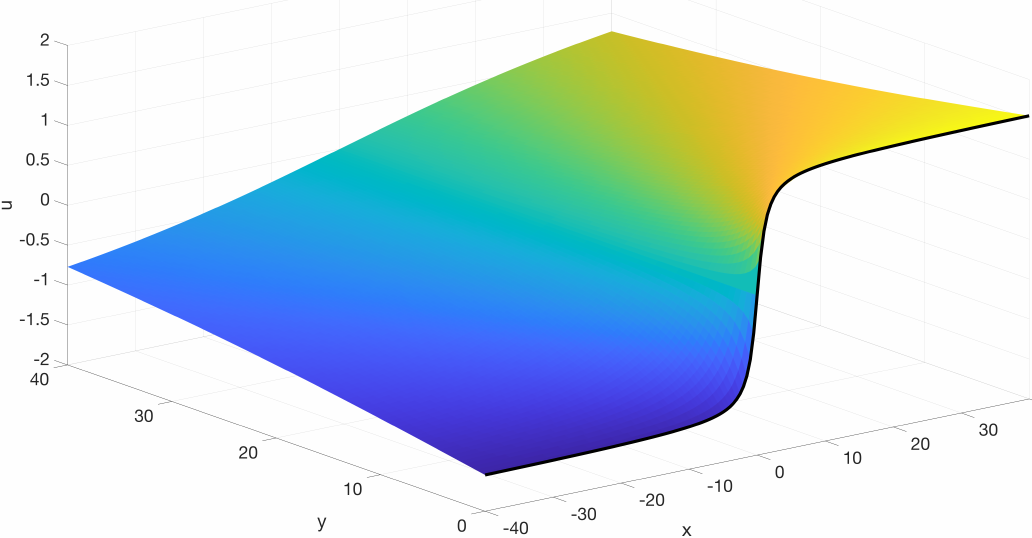} 
 \caption{(Left) Illustration of the elastic bulk continua and the atomic misfit structure at the slip plane $\Gamma$ in 2D Peierls-Nabarro dislocation model; (Right) The typical steady
transition profile for   $\phi(x,y)$ in \eqref{SS} connecting bi-states $\pm1$ on $\Gamma$.}\label{fig1}
\end{figure}

\subsection*{Metastable transition profile.}
Let us first observe a special equilibrium profile when $W$ takes the special periodic form showing periodic lattice property for crystal materials, i.e., $W(u)= \frac{1}{\pi^2}(1+\cos(\pi u))$.  The metastable steady solution, unique up to translation in $x$ direction, to 
\begin{equation}\label{phi}
\begin{aligned}
\Delta \phi = 0, \quad y>0,\\
\pt_y \phi = W'(\phi), \quad y=0,
\end{aligned}
\end{equation}
with bi-states condition $\lim_{x\to \pm\8}\phi(x,0)=\pm 1$,   is given by \cite[Lemma 2.1]{CS05}
\begin{equation}\label{SS}
\phi(x,y)= \frac{2}{\pi} \arctan \frac{x}{y+1};
\end{equation}
see Fig \ref{fig1} (Right).
We can easily calculate the derivatives of $\phi$ as
\begin{equation}
\pt_x \phi(x,y) = \frac{2}{\pi} \frac{y+1}{(y+1)^2+ x^2}, \quad \pt_y \phi(x,y)= -\frac{2}{\pi} \frac{x}{x^2+(y+1)^2}.
\end{equation}
Importantly, it has been proved in \cite{CS05, PSV13}, for a general $W\in C^{2,\alpha}(\bR)$ satisfying \eqref{potential},  there exists a unique (up to translation in $x$ direction) metastable steady solution $\phi\in C^{2,\alpha}(\bR^2_+)$ to \eqref{phi} such that 
\begin{equation}\label{phiP}
\lim_{x\to \pm\8}\phi(x,0)=\pm 1, \quad \text{ and } \,\, \pt_x \phi(x, 0)>0.
\end{equation}
In \cite[Theorem 1.6]{CS05}, they recovered the far field decay rate of the metastable profile $\phi$
\begin{equation}\label{decayR}
\phi(x,0) \sim \pm 1 - \frac{c}{x}\quad  \text{ as }x\to \pm \8
\end{equation}
with some constant $c>0$ and 
\begin{equation}\label{decayDi}
|\nabla \phi (x,y)|\leq \frac{c}{1+\sqrt{x^2+y^2}}, \quad y\geq 0, \, x\in \bR.
\end{equation}
From now on, $c>0$ is a generic constant whose value may change from line to line.

\subsection*{Main result and approach}
Notice the far field bi-states condition $\lim_{x\to \pm\8}\phi(x,y)=\pm 1$ for the matastable  equilibrium itself is not uniformly in $y$. It suggests we impose the following assumptions on the initial data $u_0(x,y)$. Denote $\bR^2_+:=\{(x,y)\in \bR^2; y\geq 0\}$ and denote $C_b(\bR^2_+)$ as the space of bounded functions that continuous up to the boundary.
\begin{asm}\label{asm1}
Let $\phi(x,y)$ be the unique solution (up to translation in $x$) to \eqref{phi}. Assume there exist constants $\xi_1, \xi_2>0$ and a function $q_0(x,y)>0$ such that
\begin{enumerate}[(i)]
\item $q_0$ satisfies
\begin{equation}
\lim_{x\to \pm\8} q_0(x,y) = 0 \quad \text{ uniformly in }y, \quad q_0(x,0)\in L^1(\bR)\cap L^\8(\bR);
\end{equation}
\item the initial data $u_0(x,y)$ satisfies $|u_0(x,y)|<1$,
\begin{equation}
\phi(x+\xi_2, y) -q_0(x,y) \leq u_0(x,y) \leq \phi(x-\xi_1, y) + q_0(x,y).
\end{equation}
\end{enumerate}
\end{asm}

In the following theorem, we state the main result in this paper, i.e.,  the uniform convergence of the dynamic solution  to its metastable equilibrium.
\begin{thm}\label{mainthm}
Suppose the initial data $u^0$ satisfies Assumption \ref{asm1}. Then the dynamic solution $u(t,x,y)$ to \eqref{maineq} converges to the static solution $\phi(x,y)$ to \eqref{phi} in the sense that for any $\eps>0$, there exist $x_0$ and $T$ such that for any $t>T$ 
\begin{equation}
|u(t, x,y)-\phi(x-x_0,y)|<\eps \quad \text{ uniformly for } (x,y)\in \bR^2_+.
\end{equation}
\end{thm}

The key point in the proof is to construct a supersolution and a subsolution to the full dynamics \eqref{maineq} so that the full dynamics can be eventually controlled and uniformly converges to the static profile $\phi(x,y)$ to \eqref{phi}. This generic method of constructing super/subsolutions to study the global stability was first proposed in the pioneering work of \cite{FM} for the classical 1D Allen-Cahn equation with double well potential. However, without quasi-static assumption, one can not reduce the full dynamics into a 1D reaction-diffusion equation. Instead, the stability of the full 2D system is only provided by an interfacial double well  potential $W$ on $\Gamma$. That is to say, the interface reaction turns on its effect to the bulk $y>0$ only through the Neumann boundary condition $\pt_y u$, which does not have a time-independent Dirichlet to Neumann map.    Our construction of super/subsolutions relies on a time decay estimate for the linearized solution $q(x,y,t)$ to \eqref{maineq} (see \eqref{q-eq}), where the linearized system is also a bulk-interface interactive system.  With an initial perturbation $q_0>0$, whose bulk part and boundary part are both nonzero, our strategy is to leverage the heat kernel in 2D and estimate its impact to the boundary $\Gamma$ through the normal derivative $\pt_y q$. To do so,  we properly decompose the linearized system as two 2D heat equations with  different dynamic boundary conditions, which are coupled only through the boundary condition on $\Gamma$; see Lemma \ref{q-es} for the algebraic decay estimate for $q(t,x,y)\lesssim \frac{1}{1+t^{\frac32}}$. Compared with the  1D reaction-diffusion equation, we do not expect any exponential decay estimate for linearized solution $q$ because there is no spectral gap and the corresponding linear operator is not self-adjoint; see Section \ref{app2}. The construction of super/subsolutions for bulk-interface coupled dynamics is inspired by a series of works  \cite{BerestyckiRR_2013, BerestyckiCRR_2015, BerestyckiRR_2016} on the diffusion equation with a Fisher-KPP type reaction in the bulk and meanwhile is influenced by  a fast diffusion line on the boundary.
However, the double-well reaction $W'(u)$ in our model \eqref{maineq} presents  only on the lower dimensional interface (the slip plane $\Gamma$ of dislocations), which makes the uniform convergence more difficult because one wants to trap the whole half-plane dynamics using only reactions on the boundary.

\subsection*{State of the art}
To study mechanical behaviors of materials with the presence of dislocations, characterization of equilibrium profile and  the dynamic process, and also the corresponding relaxation rate to the equilibrium state have proceeded in various routes at the level of mathematical analysis.
In \cite{CS05}, \textsc{Cabr\'e and Sol\`a-Morales} established the existence  and the uniqueness (up to translations) of  monotonic solutions and  also proved the metastable profile is a local minimizer of the corresponding free energy
\begin{equation}\label{totalE}
E( {u}):=\frac 1 2\int_{\bR^2_+} |\nabla  {u}|^2\ud x\ud y+\int_\Gamma W( {u})\ud x.
\end{equation}
Later, \cite{PSV13}  directly proved the existence of the reduced nonlocal equation $(-\Delta)^{\frac12}{u}|_\Gamma=-W'({u})$ on $\Gamma$, which is derived via the Dirichlet to Neumann map $\pt_n u  =(-\Delta)^{\frac12}{u} |_\Gamma$.  Very recently, in \cite{dong2021existence}, the author proved the rigidity for a class of 3D vectorial dislocation model, which  states the equilibrium profile has to be a 1D profile with uniform displacement in $z$-direction. Moreover, using an elastic extension, in \cite{gao2021revisit} the author   rigorously connect  the 2D Lam\'e system with the nonlinear boundary condition   to the 1D reduced nonlocal equation.
For the dynamics of dislocations, existing results only work under a quasi-static assumption for the elastic bulks in $y>0$,  so that one can still use the Dirichlet to Neumann map to reduce the quasi-static dynamics as a 1D nonlocal reaction-diffusion equation
\begin{equation}\label{quasi}
\pt_t u + (-\Delta)^{\frac12}{u}=-W'({u}) \quad  \text{ on } \Gamma.
\end{equation}
 In \cite{GL20long}, the long time behavior  of the single edge dislocation and its exponential relaxation was proved via a new notion of $\omega$-limit set. At a macroscopic scale, the 1D slow motion of $N$-dislocations was studied in \cite{Mon1}; see also general cases including collisions of dislocations with different orientations in \cite{patrizi2015crystal, PV1, PV2} and for    general fractional Laplacian $(-\Delta  )^{\frac s2} \tilde{u}$ with $0<s<2$ in \cite{DPV, DFV}.
 These 1D results for the quasi-static model are  motivated by the pioneering works \cite{FM, Pego, bates1, Chen_2004} for the classical 1D local Allen-Cahn equation. For bulk-interface dynamics such as the Fisher-KPP diffusion-reaction coupled with an interfacial diffusion, \cite{BerestyckiRR_2013, BerestyckiCRR_2015, BerestyckiRR_2016, rossi2017effect, berestycki2020influence} studied the propagation of fronts, which is closely related to our bulk-interface dynamics but with KPP type reaction presenting in the bulk instead of on the interface.

In the remaining of this paper,   we first briefly explain in Section \ref{app2} why a solution to \eqref{maineq} exists,  then the asymptotic behavior is described in the remaining three sections. We prove the key estimates for the construction of supersolutions/subsolutions in Section \ref{sec3}. These rely on a decomposition for dynamic boundary condition and a heat kernel computation, which is developed in Section \ref{sec4}. Then in Section \ref{sec3.2}, we complete the proof of the main convergence result, Theorem \ref{mainthm}. The  gradient flow derivations for the bulk-interface dynamics \eqref{maineq} is shown in Appendix \ref{app1}.
 
\section{$C_0$-semigroup and existence of dynamic  solution }\label{app2}
As a preliminary,  in this section, we first clarify the linear operator for \eqref{maineq} is not self-adjoint operator and  the theory of $C_0$-semigroup solution for semilinear equations ensures the existence of a dynamic solution to \eqref{maineq}. 

 Without the nonlinear term $W$, regarding \eqref{maineq} as a Kolmogorov forward equation, \textsc{Feller}  explicitly characterized the generator $L$ with domain $D(L)$ of a contraction semigroup  in  \cite{Feller_1952, Feller_1954}; see also \cite{V59}. These pioneering works in 1950s first  identified all admissible boundary conditions for a second order differential operator $L$ to generate a contraction semigroup on a properly chosen Banach space. Denote $C(\overline{\mathbb{R}^{2}_+})$ as the space of continuous functions $u$ in $\{(x,y)\in\bR^2; y\geq 0\}$ such that there exists finite limit for $u$ at far fields. For any test function $f\in C_b^2(\overline{\mathbb{R}^{2}_+})$, 
$L : D(L )\subset C(\overline{\mathbb{R}^{2}_+}) \to C(\overline{\mathbb{R}^{2}_+})$ is defined as
\begin{equation}
L  f:=\left\{ \begin{array}{c}
\Delta f, \quad x\in \mathbb{R}^{2}_+,\\
- \pt_n f, \quad x\in \Gamma
\end{array}\right.
\end{equation}
with the domain  $$D(L )=\{f\in C_b^2(\overline{\mathbb{R}^{2}_+}); ~\Delta f= - \pt_n f \text{ for }x\in \Gamma\}.$$
Then  the full dynamics \eqref{maineq} in a matrix form  will be 
\begin{equation}\label{bweq}
\pt_t \bbs{
\begin{array}{c}
u\\
u_\Gamma
\end{array}
} =L u + \bbs{
\begin{array}{c}
0\\
-W'(u)
\end{array}
} =: \mathcal{L} u, \quad u(x,0)=u_0(x).
\end{equation}
Although $L $ is symmetric, it is not self-adjoint and the spectral analysis is more delicate. Indeed, the adjoint operator $L^*: D(L)\subset L^1(\bR^2_+) \times L^1(\Gamma)\to  L^1(\bR^2_+) \times L^1(\Gamma)$ is given by
\begin{equation}
L^*\rho:=\left\{ \begin{array}{c}
\Delta \rho, \quad x\in \mathbb{R}_+^2,\\
-\pt_n \rho, \quad x\in \Gamma.
\end{array}
\right.
\end{equation}   
The domain of the adjoint $L ^*$ has been characterized by \cite{Feller_1952}, 
\begin{equation}
D(L ^*)= \{(\rho, \rho_\Gamma)\in L^1(\bR^2_+) \times L^1(\Gamma); ~ \rho\in W^{2,1}(\bR^2_+) \}.
\end{equation}
 Here we notice for dimension $n=2$, the trace theorem implies $W^{2,1}(\bR^2_+)\hookrightarrow L^q(\Gamma)$ for any $1\leq q<+\8$. Thus the trace $\rho_\Gamma$ is well-defined.

There are many results  on that the linear operator $L$ generates a strongly continuous semigroup on the product space $L^1(\bR^2_+)\otimes L^1(\Gamma)$ or on $C(\overline{\mathbb{R}^{2}_+})\otimes C(\Gamma)$. For instance, the associated Feller semigroup on $C(\overline{\mathbb{R}^{2}_+})$ was studied in \cite{Feller_1952, Feller_1954}; see also \cite{Arendt_Metafune_Pallara_Romanell_2003, Engel_2003} for bounded  domain  $\Omega$ and  see detailed investigations in \cite[Proposition 9, Theorem 7]{Guidetti_2016} for  half space $\overline{\mathbb{R}^{2}_+}.$ Since $W(\cdot)\in C_b^3(\bR)$, the semigroup solution for the quasilinear one \eqref{bweq} can be obtained by proving that for some $\omega>0$ large enough, $\mathcal{L}-\omega I$  is the generator of a strongly continuous semigroup of contraction. 
Although the existence is not the focus of this paper, we refer to \cite{FGGR_2000, VV_2008, Xiao_Liang_2008, MKR, gao2018new} for existence results of various kinds of dynamic boundary condition problems.

\section{Construction of supersolutions and subsolutions}\label{sec3}

In this section, we will  give the crucial supersolution/subsolution estimates for the dynamic solution in Proposition \ref{prop1}. This relies on a detailed decay estimate for the corresponding linearized bulk-interface dynamics; see Lemma \ref{q-es}. Compared with the classical local/nonlocal diffusion-reaction equation, the decay rate w.r.t. time $t$ becomes an algebraic rate due to the bulk-interface interaction.

Recall the full dynamics \eqref{maineq}
\begin{equation*}
\left\{ \begin{array}{c}
\pt_t u - \Delta u=0, \quad y>0;\\
\pt_t u - \pt_y u + W'(u)=0, \quad y=0.
\end{array} \right. 
\end{equation*}
Let $\phi(x,y)$ be the metastable equilibrium solution to \eqref{phi}. Then we know $\phi$ satisfies \eqref{phiP}-\eqref{decayDi}.
The goal of this section is using $\phi$ to construct  supersolution and subsolution to \eqref{maineq} so that the dynamic solution is approximately  squeezed in between two static transition profiles provided time is large enough.

Recall the double well potential $W$ satisfies \eqref{potential}. This interfacial misfit potential $W(\cdot)$ on slip plane $\Gamma$ has two local minimums $\pm1$, which essentially determines and drives the whole dislocation dynamics to the metastable transition profile $\phi(x,y)$.  Let us first clarify some properties of  $W$.  Since $W''(\pm 1)>0$,
there exist constants $\mu>0$,
$\delta>0$ such that,
\begin{equation}\label{con_mu}
\begin{aligned}
W'(\phi+q)-W'(\phi)\geq \mu q\quad &\text{ for  }1-\delta \leq \phi \leq 1, \, 0<q<\delta,\\
W'(\phi+q)-W'(\phi)\geq \mu q \quad &\text{ for }-1 \leq \phi \leq -1+\delta, \, 0<q<1-\delta.
\end{aligned}
\end{equation}
Moreover, for $\phi\in[-1+\delta, 1-\delta]$, there exist constants $k>0$, $\beta\geq 0$ such that
\begin{equation}\label{con_k}
 |W'(\phi-q)-W'(\phi)|\leq kq ,
\end{equation}
 and by \eqref{phiP},
\begin{equation}\label{con_beta}
  \pt_x\phi(x, 0) \geq \beta>0 \quad \text{for } x \text{ such that } \phi(x)\in[-1+\delta, 1-\delta].
\end{equation}

Next, In Proposition \ref{prop1}, we state the crucial comparison principle for the dynamic solution to \eqref{maineq}. The proof of this proposition relies on the properties of the interfacial potential and  Lemma \ref{q-es} on the decay estimate of the linearized bulk-interface dynamics. Lemma \ref{q-es} ensures one can trap the full dynamics using merely an interfacial double-well potential. We will first give the proof of Proposition \ref{prop1} in this section, and then give the proof of Lemma \ref{q-es} in the next section.

\begin{prop}[Construction of supersolutions/subsolutions]\label{prop1}
Let $u(t,x,y)$ be the solution to bulk-interface dynamics \eqref{maineq} with initial data $u_0(x,y)$.
Suppose the initial data $u_0(x,y)$ satisfies 
Assumption \ref{asm1}.
Then there exist constants $\xi_1, \xi_2$, $C>0, M>0$ and $q_0(x,y)>0$, such that
\begin{equation}\label{supsub}
\phi(x-\xi_2+2Mt^{-\frac12},y)-\frac{q_0(x,y)}{1+C t^{\frac32}} \leq u(t,x,y) \leq \phi(x-\xi_1+2Mt^{-\frac12},y)+\frac{q_0(x,y)}{1+C t^{\frac32}},
\end{equation}
where $\phi(x,y)$ is the steady profile  satisfying \eqref{phi}.
\end{prop}
\begin{proof}
We will construct a supersolution as
\begin{equation}
\bar{u}(t,x,y):=\min\{ 1,  \phi(x+\xi(t), y) + q(t,x,y)\} \in [-1,1]
\end{equation}
by choosing $\xi(t)$ and $q(t,x,y)\geq 0$. The construction of subsolution is similar.

Step 1:
for $u_0$ satisfying  Assumption \ref{asm1}, there exists a number $\xi_1$ such that
\begin{equation}
u_0(x,y)\leq \phi(x-\xi_1, y)+ q_0(x,y),
\end{equation}
and $0<q_0< 1$, $\lim_{x\to \pm\8} q_0(x,y)=0$ uniformly in $y$.

 From the steady solution to \eqref{phi}, we know $\Delta \phi=0$ for $y>0$ and $\pt_y \phi = W'(\phi)$ on $y=0$. Therefore we obtain
\begin{align}\label{eq_sup}
\begin{array}{ll}
&\pt_t \bar{u} - \Delta \bar{u}
=\pt_x\phi(x+\xi(t),y) \xi' +\pt_t q -\Delta q \qquad \text{ for } y>0;\\
&\pt_t \bar{u} - \pt_y \bar{u} + W'(\bar{u})\\
& \qquad =\pt_x\phi(x+\xi(t),0)\xi'+ \pt_t q - \pt_y q - W'\big(\phi(x+\xi(t),0)\big) + W'\big(\phi(x+\xi(t),0)+q\big) \quad \text{ for } y=0.
\end{array}
\end{align}  
Here the function $\xi(t)$ with $\xi'\geq 0$ will be chosen explicitly in Step 3. Now we choose $q(t,x,y)$ with $0<q_0(x,y)<1$ such that
\begin{equation}\label{q-eq}
\begin{aligned}
\pt_t q - \Delta q =0, \quad y>0,\\
\pt_t q - \pt_y q + \mu q = 0, \quad y=0,
\end{aligned}
\end{equation}
where $\mu>0$ is the constant in \eqref{con_mu}.

Step 2: to prove $\bar{u}$ is a supersolution, divide the space into several sets
$$
  I_1:=\{(t, x, y); ~\phi(x+\xi(t),y) \in [1-q, 1] \},
  $$
  $$
  I_2:=\{(t, x, y); ~\phi(x+\xi(t),y) \in [1-\delta, 1-q] \},
  $$
  $$
  I_3:=\{(t, x, y); ~\phi(x+\xi(t),y) \in [-1+\delta, 1-\delta] \},
  $$
  $$
  I_4:=\{(t, x, y); ~\phi(x+\xi(t),y) \in [-1, -1+\delta] \}.
  $$
  Here, some sets being empty are allowed. We  now  estimate the right-hand-side of \eqref{eq_sup} for both the bulk and the interface in all possible cases as follows. 
  
  Case (i): if $(t, x, y)\in I_1$, then $\phi(x+\xi(t),y)+q(t,x,t)\geq 1$ and $\bar{u}\equiv 1.$
  
  Case (ii): if $(t, x, y)\in I_2$ or $(t, x, y)\in I_4$,
 then  from \eqref{con_mu}
\begin{equation}
- W'(\phi(x+\xi(t))) + W'(\phi(x+\xi(t))+q) \geq \mu q, \quad \text{ for some }\mu>0.
\end{equation}
Thus using \eqref{q-eq}, we have on $y=0$,
\begin{equation}
\pt_t \bar{u} - \pt_y \bar{u} + W'(\bar{u})\geq \pt_x\phi(x+\xi(t),y)\xi'+ \pt_t q - \pt_y q+ \mu q=\pt_x\phi(x+\xi(t),y)\xi'=: R_1,
\end{equation}
and for $y>0$
\begin{equation}
\pt_t \bar{u} - \Delta \bar{u}=\pt_x\phi(x+\xi(t),y) \xi' +\pt_t q -\Delta q=R_1.
\end{equation}
Here we used the equations for $q$ in \eqref{q-eq}.
Since the profile $\phi$ is increasing w.r.t $x$, so we know $R_1\geq 0$  and  conclude
\begin{equation}
\begin{aligned}
\pt_t \bar{u} - \Delta \bar{u}\geq 0 , \quad y>0;\\
\pt_t \bar{u} - \pt_y \bar{u} + W'(\bar{u})\geq 0, \quad y=0.
\end{aligned}
\end{equation}

Case (iii), if $(t,x,y)\in I_3$, then from \eqref{con_k} we have on $y=0$
\begin{equation}
\pt_t \bar{u} - \pt_y \bar{u} + W'(\bar{u})\geq \beta \xi'+ \pt_t q - \pt_y q- k q=: R_2.
\end{equation}
Since $q$ satisfies \eqref{q-eq}, $R_2\geq 0$ if and only if
\begin{equation}\label{xi_1}
\xi'\geq \frac{-\pt_t q + \pt_y q + kq}{\beta}= \frac{(\mu+k)q}{\beta}.
\end{equation}
Therefore, to prove $\bar{u}$ is a supersolution, we only need the following lemma to estimate the decay of $q(t,x,0)$. The proof of Lemma \ref{q-es} will be given later in Section \ref{sec4}.
\begin{lem}\label{q-es}
Let $q$ be the solution to \eqref{q-eq} with initial data $0<q_0<1$, then there exists $C>0$ such that
\begin{equation}\label{q_decay}
|q(t,x,y)|\leq \frac{q_0(x,y)}{1+C t^{\frac32}}.
\end{equation}
\end{lem}
As a consequence of this lemma and \eqref{xi_1}, we can choose $\xi(t)=\xi_1-2M t^{-\frac12}$ satisfying $\xi'(t)= M{t^{-\frac32}}$ with $M>0$ large enough such that \eqref{xi_1} holds.   Thus we obtain $\bar{u}$ is a supersolution and conclude \eqref{supsub}. The construction of subsolution is similar and we omit the details. 
\end{proof}

From the proof of Proposition \ref{prop1}, we see the construction relies on a time decay estimate for the linearized bulk-interface coupled dynamics \eqref{q-eq}.  

\section{The proof of Lemme \ref{q-es}} \label{sec4}
In this section, we give the proof of the key decay estimates in Lemma \ref{q-es} for the linearized system \eqref{q-eq}. The proof consists of (i) decomposing the bulk-interface coupled linear system as two heat equations with different dynamic boundary conditions; (ii) estimating the boundary stabilization rate for each sub-problems.

Step 1. Since the initial data $q_0$ has both nonzero bulk part and interface part, we first choose a proper decomposition of $q$ to decouple the dynamics for $y>0$ and $y=0$. 

Assume the initial data $q_0(x,y)$ can be expressed as bulk part and interface part
\begin{equation}
q_0(x,y)= q_0 \mathbb{\chi}_{y>0}  + q_0 \mathbb{\chi}_{y=0}  =: q_0^b(x,y) + q_0^s(x),
\end{equation}
where $\mathbb{\chi}$ is the characteristic function.
We do the Laplace transform of $q(t,x,y)$ with respect to $t$ using Laplace variable $\lambda$, the Fourier transform with respect to $x$ using Fourier variable $\eta$ and denote it as $\hat{q}:=\hat{q}(\lambda, \eta, y).$
Then $\hat{q}$ satisfies
\begin{equation}\label{q-z}
\begin{aligned}
\lambda \hat{q}-\pt_{yy}\hat{q}+ \eta^2 \hat{q}=\hat{q}_0^b,\quad y>0,\\
\lambda \hat{q}- \pt_y \hat{q} + \mu \hat{q} = \hat{q}_0^s, \quad y=0,
\end{aligned}
\end{equation}
where $\hat{q}_0^b=\hat{q}_0^b(\eta,y)$ (resp. $\hat{q}_0^s=\hat{q}_0^s(\eta)$) is the Fourier transform of $q_0^b(x,y)$ (resp. $q_0^s(x)$) with respect to $x$. Since the equations for $q$ are linear, we can construct $q=q_1+q_2$ with $\hat{q}=\hat{q}_1+ \hat{q}_2$ such that $\hat{q}_1$ satisfies
\begin{equation}\label{q1}
\begin{aligned}
\lambda \hat{q}_1-\pt_{yy}\hat{q}_1+ \eta^2 \hat{q}_1=0,\quad & y>0,\\
\lambda \hat{q}_1- \pt_y \hat{q}_1 + \mu \hat{q}_1 = \hat{q}_0^s+ \pt_y \hat{q}_2, \quad & y=0,
\end{aligned}
\end{equation} 
while $\hat{q}_2$ satisfies
\begin{equation}\label{q2}
\begin{aligned}
\lambda \hat{q}_2-\pt_{yy}\hat{q}_2+ \eta^2 \hat{q}_2=\hat{q}_0^b,\quad & y>0,\\
 \hat{q}_2(\lambda, \eta, 0)=0, \quad & y=0.
\end{aligned}
\end{equation}
One can first solve $\hat{q}_2$ and then $\hat{q}_1$ with the dynamic boundary input from $q_2$.

Step 2. Estimate  $q_1$ and $q_2$ separately.

First, for solution $\hat{q}_2$ satisfying \eqref{q2}, one can directly estimate the solution $q_2$ to the heat equation for $y>0$ with initial data $q_0^b$ and  Dirichlet boundary condition $q_2(t,x,0)=0$ on $y=0$. Denote $\bar{q}_0^b$ as the odd extension of $q_0^b$ to the whole space $\mathbb{R}^2$ and denote $\Phi(t,x,y)= \frac{1}{4\pi t}e^{-\frac{x^2+y^2}{4t}}$ as the 2D fundamental solution to heat equation in $\bR^2$. Then the solution formula for $q_2$ is given by $q_2 = \Phi * \bar{q}_0^b$.  Since $|\pt_y \Phi|\leq \frac{cy}{t^2}e^{-\frac{x^2+y^2}{4t}}$, then using the change of variable $s=\frac{y^2}{4t}$, we obtain
\begin{equation}\label{pt_y}
|\pt_y q_2|\leq c \int_0^{+\8}\int_{\mathbb{R}} |\pt_y \Phi| \ud x \ud y \leq c\int_{\mathbb{R}} \frac{1}{t^{\frac12}} e^{-\frac{x^2}{4t}} \ud x \int_0^{+\8} \frac{1}{t^\frac12} e^{-s} \ud s  \leq   \frac{c}{t^{\frac12}}.
\end{equation}

Second, for solution $\hat{q}_1$ satisfying \eqref{q1}, we seek solution with far field decay as $y\to +\8$
\begin{equation}
\hat{q}_1(\lambda, \eta, y) = \hat{q}_1(\lambda, \eta, 0) e^{-\sqrt{\lambda+\eta^2} y}.
\end{equation}
From \eqref{q1} and \eqref{q2}, we know $q_1$   satisfies the boundary condition
\begin{equation}\label{q1_22}
\pt_y \hat{q}_1(\lambda, \eta, 0) = - \sqrt{\eta^2 + \lambda} \hat{q}_1(\lambda, \eta, 0) = (\lambda+ \mu) \hat{q}_1(\lambda, \eta, 0)- \hat{q}_{2in}^s(\lambda,\eta),
\end{equation}
where 
\begin{align}\label{def_q2in}
\hat{q}_{2in}^s(\lambda,\eta):=\hat{q}_0^s(\eta)+ \pt_y\hat{q}_2(\lambda, \eta, 0),\qquad 
q^s_{2in}(t,x):= q_0^s(x) + \pt_y q_2(t,x,0).
\end{align}
Therefore 
\begin{equation}
\hat{q}_1(\lambda, \eta, 0) = \frac{\hat{q}^s_{2in}(\lambda,\eta)}{\lambda+ \mu + \sqrt{\lambda+\eta^2}}.
\end{equation}
Then by inverse transform we obtain
\begin{equation}\label{tm24}
q_1(t,x,0) = c q_{2in}^s(t,x) * \int_\Upsilon \int_{\mathbb{R}} e^{\lambda t} \frac{e^{i x \eta}}{\lambda+ \mu + \sqrt{\lambda+\eta^2}} \ud \eta \ud \lambda,
\end{equation}
where $\Upsilon$ is the  vertical line from $Re(\lambda)-T i$ to $Re(\lambda)+T i$ with $T\to +\8$  and $Re(\lambda)$ is greater than any singularities of $\int_{\mathbb{R}}\frac{e^{i x \eta}}{\lambda+ \mu + \sqrt{\lambda+\eta^2}} \ud \eta $. 

Third, fixing $t>0$, we now estimate the $L^\8(\bR)$ norm  for $q_1(t,x,0)$.

Observe that the branch point $\lambda=-\eta^2$ is a singularity. Then take the branch cut $(-\8, -\eta^2)$ and set $\lambda=-\eta^2 + r e^{i\theta}$ with $-\pi\leq \theta<\pi,\,r>0$. We calculate the roots of $\lambda+ \mu + \sqrt{\lambda+\eta^2}=0$ for $r>0$. It is sufficient to solve
\begin{equation}\label{agl}
\begin{aligned}
\mu - \eta^2 + r \cos \theta + \sqrt{r}\cos \frac{\theta}{2}=0,\\
r \sin \theta + \sqrt{r}\sin \frac{\theta}{2}=0.
\end{aligned}
\end{equation} 
For $-\pi\leq \theta<\pi$, \eqref{agl} indeed has no solution if $\eta^2< \mu$ and has only one solution if $\eta^2\geq \mu$
\begin{equation}\label{r}
\theta=0,\quad r+ \sqrt{r} = \eta^2 - \mu. 
\end{equation}
This shows that we can take $\Upsilon$ as any vertical line such that 
\begin{equation}\label{real}
0>Re(\lambda)>r^*-\eta^2= -\sqrt{r^*}-\mu,
\end{equation}
where $r^*$ is the solution to \eqref{r}.

To estimate the complex integral  $\int_\Upsilon \int_{\mathbb{R}} e^{\lambda t} \frac{e^{i x \eta}}{\lambda+ \mu + \sqrt{\lambda+\eta^2}} \ud \eta \ud \lambda$, we only need to consider two cases, i.e.,  Case (i) $\eta^2\geq  \mu$ and Case (ii) $\eta^2< \mu$.  For Case (ii), since there is no singularity in the complex integral,  by Cauchy's integral theorem and standard calculations,   it is sufficient to estimate the complex integral    for one piece of the contour, i.e. 
$$C_1:=\{\lambda=-\eta^2 + \eps e^{i\theta}, \quad -\frac{\pi}{2}<\theta<\frac{\pi}{2}\}.$$
For Case (i), to avoid the singularity at $\lambda=-\eta^2 + r^*$,  it is sufficient to estimate the complex integral for one piece of the contour, i.e. $$C_2:=\{\lambda=-\mu + \eps e^{i\theta},\quad -\frac{\pi}{2}<\theta<\frac{\pi}{2}\}.$$
Below, we explain detailed estimates for these two cases.

 Case (i). For $\eta^2\geq\mu$, from \eqref{real}, we can always choose  $Re(\lambda)\ll -\eps$, which gives the exponential decay w.r.t $t$
\begin{equation}\label{case1es}
\begin{aligned}
&\left|\int_{C_2} \int_{\mathbb{R}} e^{\lambda t} \frac{e^{i x \eta}}{\lambda+ \mu + \sqrt{\lambda+\eta^2}} \ud \eta \ud \lambda \right|_{L_x^1(\bR)}\\
 \leq& c e^{-\mu t }  \left| \int_{\bR} e^{ -( \sqrt{\alpha t} \eta - \frac{i x}{2\sqrt{\alpha t}} )^2} e^{\alpha \eta t} e^{-\frac{x^2}{4\alpha t}} \ud \eta \right|_{L_x^1(\bR)}  \leq  c e^{-\frac{\mu}{2}t}
\end{aligned}
\end{equation}
for $\alpha >0$ small enough and $t>0$ large enough.
 Thus for any $t>0$ large enough,  from \eqref{tm24} and Young's convolution inequality, we obtain the uniform estimate for $q_1$
\begin{equation}\label{q-exp}
\begin{aligned}
\|q_1(t,x,0)\|_{L^\8}\leq &  \left\|q_{2in}^s(t,x)\right\|_{L_x^\8}\left\|\int_\Gamma \int_{\mathbb{R}} e^{\lambda t} \frac{e^{i x \eta}}{\lambda+ \mu + \sqrt{\lambda+\eta^2}} \ud \eta \ud \lambda\right\|_{L_x^1}\\
\leq &\left\|q_0^s(x) + \frac{1}{\sqrt{t}}\right\|_{L_x^\8}\left\|\int_\Gamma \int_{\mathbb{R}} e^{\lambda t} \frac{e^{i x \eta}}{\lambda+ \mu + \sqrt{\lambda+\eta^2}} \ud \eta \ud \lambda\right\|_{L_x^1} \leq c  e^{-ct}.
\end{aligned}
\end{equation}
Here in the second inequality, we used definition of $q_{2in}^s$ in \eqref{def_q2in} and  the estimate \eqref{pt_y}, while in the last inequality we used \eqref{case1es}.

Case (ii).  For  $\eta^2<\mu$, we shall be careful about the part $\eta^2\ll 1$ and $0>Re(\lambda)> -\eps$, otherwise we still have the exponential decay as in \eqref{q-exp}. 
Denote $\xi=\sqrt{t}\eta$ and $z=\eps e^{i\theta}$ then
\begin{equation}\label{tm329}
\begin{aligned}
|q_1(t,x,0)|\leq &  \left|c q^s_{2in}(t,x) * \frac{1}{\mu} \int_{C_1} \int_{|\eta|\leq  1} e^{-\lambda t+ i x \eta} \ud \eta \ud \lambda \right|\\
=&  \left|c q^s_{2in}(t,x) * \frac{1}{\mu} \int_{-\frac{\pi}{2}}^{\frac{\pi}{2}} \int_{|\eta|\leq  1} e^{-\eta^2 t + \eps e^{i\theta} t+ i x \eta} \ud \eta \ud \theta \right|\\
\leq& \left| \frac{c}{\mu}  \frac{1}{t^{\frac32}} q_{2in}^s(t,x) *\int_{\mathbb{R}} e^{-\xi^2 + i \frac{x\xi}{\sqrt{t}}} \ud \xi \right|  \\
\leq& \left|\frac{c}{t^{\frac32}} { (q_{2in}^s(t,x) * e^{-\frac{x^2}{2t}} )} \int_{\mathbb{R}} e^{-\frac{(\xi-\frac{x i }{\sqrt{2t}})^2}{2}}\ud \xi \right|,
\end{aligned}
\end{equation}
where the factor $\frac{1}{t^{\frac32}}$ comes from the change of variable $\xi = \sqrt{t} \eta$ and $\tilde{z}=z t$.
Then from the definition of $q_{2in}^s$ in \eqref{def_q2in}, the estimate \eqref{pt_y} and  Young's convolution inequality, we obtain
\begin{equation}\label{q1est}
\begin{aligned}
\|q_1(t,x,0)\|_{L^\8}\leq& \frac{c}{t^{\frac32}} \Big\| [q_{0}^s(x)+ \pt_y q_2(t,x,0) ]* e^{-\frac{x^2}{2t}} \Big\|_{L_x^\8} \\
  \leq&  \frac{c}{t^{\frac32}} \Big\| q_{0}^s(x) * e^{-\frac{x^2}{2t}}+ \frac{1}{\sqrt{t}} * e^{-\frac{x^2}{2t}}\Big\|_{L_x^\8} \\
 \leq & \frac{c}{t^{\frac32}} \bbs{  \|q_0^s\|_{L^1_x} + \frac{1}{\sqrt{t}} \|e^{-\frac{x^2}{2t}}\|_{L^1_x} } \leq  \frac{c}{t^{\frac32}},
\end{aligned}
\end{equation}
where we used $\|e^{-\frac{x^2}{2t}}\|_{L^1_x}\leq c t^{\frac12}$.

Step 3. Combine estimates for $q_1$ and $q_2$ to estimate $q(t,x,0)$ and $q(t,x,y)$.

First from \eqref{q1_22} and $\hat{q}_2(\lambda, \eta, 0)=0$, we have on $y=0$
\begin{equation}
\pt_y \hat{q} = \pt_y \hat{q}_1 + \pt_y \hat{q}_2 = - \sqrt{\eta^2 + \lambda} \hat{q}_1 + \pt_y \hat{q}_2=- \sqrt{\eta^2 + \lambda} \hat{q} + \pt_y \hat{q}_2.
\end{equation}
This, together with the boundary condition for $\hat{q}$ in \eqref{q-z}, we conclude on $y=0$
\begin{equation}
(\lambda+ \mu + \sqrt{\eta^2 + \lambda}) \hat{q} - \pt_y \hat{q}_2 = \hat{q}_0^s.
\end{equation}
Therefore from \eqref{def_q2in} and all cases discussed in Step 2,  one have for $t$ large enough, 
\begin{equation}
\|q(t,x,0)\|_{L^\8}= \left\|c q_{2in}^s *   \int_\Upsilon \int_{\mathbb{R}} e^{\lambda t} \frac{e^{i x \eta}}{\lambda+ \mu + \sqrt{\lambda+\eta^2}} \ud \eta \ud \lambda \right\|_{L^\8} \leq \frac{c}{t^{\frac32}},
\end{equation}
where $\Upsilon$ is the same path as \eqref{tm24}. Thus from the maximal principle for heat equation, we conclude the solution $q$ to \eqref{q-eq} satisfies the estimate \eqref{q_decay}.

{\begin{rem}
Suppose the assumption for the initial data changes to  Assumption II: Assume the initial data $u_0(x,y)=\phi(x-x_0,y)+q_0(x,y)$ for some $x_0$ and $q_0(x,0)$ satisfies $\|q_0(x,0)\|_{L^p_x(\mathbb{R})}\leq c$ for some $1<p<\8$.
Then \eqref{q1est} becomes
$
\|q_1(t,x,0)\|_{L^\8(\bR)}  \leq   \frac{c}{t^{1+\frac1{2p}}}.
$
Indeed, from the interpolation inequality, we know
\begin{equation}
\|e^{-\frac{x^2}{2t}}\|_{L^q_x} \leq \|e^{-\frac{x^2}{2t}}\|_{L^1_x}^{\frac1q} \|e^{-\frac{x^2}{2t}}\|_{L^\8_x}^{1-\frac1q}\leq c  t^{\frac1{2q}},
\end{equation}
for $1<q<\8$.
Then by  Young's convolution inequality, we obtain
\begin{equation}
\Big\| q_{0}^s(x) * e^{-\frac{x^2}{2t}} \Big\|_{L^\8}\leq \big\|q_0^s\big\|_{L^p}\big\|e^{-\frac{x^2}{2t}}\big\|_{L^q} \leq  c  t^{\frac1{2q}}
\end{equation}
for $\frac1p+\frac1q=1.$ Therefore
\begin{equation}
\begin{aligned}
\|q_1(t,x,0)\|_{L^\8}\leq& \frac{c}{t^{\frac32}} \Big\| q_{0}^s * e^{-\frac{x^2}{2t}} + \frac{1}{\sqrt{t}}   * e^{-\frac{x^2}{2t}} \Big\|_{L^\8}  \leq  \frac{c}{t^{1+\frac1{2p}}}.
\end{aligned}
\end{equation}
As a consequence, in \eqref{xi_1} we can choose $\xi(t)=\xi_1-2pM t^{-\frac1{2p}}$ with $\xi'(t)= M{t^{-1-\frac1{2p}}}$ and $M>0$ large enough such that \eqref{xi_1} holds.
\end{rem} 
 }

\section{Proof of Theorem \ref{mainthm}: Uniform convergence to metastable equilibrium}\label{sec3.2}

The  crucial comparison principle obtained
in Proposition \ref{prop1} helps us to control the dynamic solution in between two steady transition profiles as time becomes large enough. In this section,  combining Proposition \ref{prop1} with the energy dissipation law \eqref{energyD}, we  prove the uniform convergence of the dynamic solution of \eqref{maineq} to the equilibrium $\phi(x,y)$.

\begin{proof}[Proof of Theorem \ref{mainthm}]
Step 1. From the boundedness of initial data $u^0$ and smoothness of the nonlinear potential $W$, by the  parabolic interior regularity, the solution $u$ to \eqref{maineq} satisfies the following uniform bounds
\begin{equation}
|u|\leq c, \quad |\nabla u|\leq c, \quad |D^2 u| \leq c.
\end{equation}
Then from the Arzela-Ascoli Theorem, for any bounded set $B_k$, there exist $u^*(x,y)$ and $t_{n_k}^k$ such that as $n_k\to +\8$,
\begin{equation}
u(t^k_{n_k}, \cdot, \cdot) \to u^*(\cdot, \cdot) \quad\text{ uniformly in } B_k.
\end{equation}
Then by diagonal argument, there exists subsequence $t^\ell_\ell$ such that as $\ell\to+\8$
\begin{equation}\label{bin}
u(t^\ell_{\ell}, \cdot, \cdot) \to u^*(\cdot, \cdot) \quad\text{ uniformly in } B_\ell.
\end{equation}

Step 2. From Proposition \ref{prop1}, we know for $(x,y)\in B_\ell^c$, it also holds
\begin{equation}\label{tmC}
\phi(x-\xi_2+2Mt^{-\frac12},y)-\frac{q_0(x,y)}{1+C t^{\frac32}} \leq u \leq \phi(x-\xi_1+2Mt^{-\frac12},y)+\frac{q_0(x,y)}{1+C t^{\frac32}}. 
\end{equation}
Combining \eqref{tmC} with the uniform decay of $|\nabla \phi|$ in \eqref{decayDi}, we know there exists $N$ large enough such that for any $\ell>N$ and for any fixed $c$,
\begin{equation}
|u(t^\ell_\ell,x,y) -\phi(x-c,y)| <\eps\quad \text{ uniformly in } B^c_\ell.
\end{equation}
This, together with \eqref{bin}, shows for any $\eps>0$,
there exists $N$ large enough such that for any $\ell>N$
\begin{equation}\label{Usub}
|u(t^\ell_\ell,x,y) - u^*(x, y)|<\eps  \quad \text{ uniformly for } (x,y)\in \bR^2_+
\end{equation}
and   there exist constants $c_1,\,c_2$ such that $u^*$ satisfies
\begin{equation}
\phi(x-c_2, y)-\eps \leq u^*(x,y) \leq \phi(x-c_1, y)+\eps.
\end{equation}

Step 3. Recall the total energy $E$ defined in \eqref{totalE} for system \eqref{maineq}. Notice the energy dissipation law \eqref{energyD} for dynamic solution $u$
$$
\dot{E}(t) =- Q(t)\leq 0.
$$
Now we claim there is a lower bound for $E(t)$ provided $t$ large enough. Indeed, the first term in $E$ is positive while the second term in $E$ is bounded
\begin{equation}
\begin{aligned}
&\left|\int_\Gamma W(u) \ud x\right| \\
=& \Big|\int_{y=0, x>0} W'(1)(u-1) \ud x +\int_{y=0, x>0} W''(\xi)(u-1)^2 \ud x\\
&+ \int_{y=0, x<0} W'(-1)(u+1) \ud x +\int_{y=0, x<0} W''(\xi)(u+1)^2 \ud x\Big|\\
\leq& c \int_{y=0, x>0} (u-1)^2 \ud x + c \int_{y=0, x<0} (u+1)^2 \ud x,
\end{aligned}
\end{equation}
where we used properties for $W$ in \eqref{potential}.
Then combining the decay rate of $\phi(x,0)$ in \eqref{decayR} with \eqref{tmC}, for $u(t_\ell^\ell)$ with $t^\ell_\ell$ large enough, we know
\begin{equation}
\left|\int_\Gamma W(u) \ud x\right| \leq c \int_\Gamma \frac{1}{(1+x)^2} \ud x \leq c.
\end{equation}
This implies a lower bound for $E(t)$  provided $t$ is large enough.
Therefore there exists a subsequence (still denoted as $t_\ell^\ell$) such that when $t_\ell^\ell\to +\8$, we must have $Q(t_\ell^\ell)\to 0.$ Thus we know
\begin{equation}
\begin{aligned}
\liminf_{t_\ell^\ell \to 0} \int_{\bR^2_+} |\Delta u|^2 \ud x \ud y \leq\liminf_{t_\ell^\ell \to 0} Q(t^\ell_\ell) = 0,\\
\liminf_{t_\ell^\ell \to 0} \int_{\Gamma} |\pt_y u - W'(u)|^2 \ud x  \leq\liminf_{t_\ell^\ell \to 0} Q(t^\ell_\ell) = 0.
\end{aligned}
\end{equation}
This, together with \eqref{Usub} and Fatou's lemma,  yields the limit $u^*$ satisfies the static equation \eqref{phi}. From the uniqueness (upto translation) of static problem \eqref{phi}, we know $u^*(x,y)=\phi(x-x_0, y)$ for some $x_0.$

Step 4. Combining with the small data stability due to Proposition \ref{prop1}, we know for any $\eps>0$ there exists $T$ such that for any $t>T$
\begin{equation}
|u(t, x,y)-\phi(x-x_0,y)|<\eps \quad \text{ uniformly for }(x,y)\in \bR^2_+.
\end{equation}
This concludes the uniform convergence result in  Theorem \ref{mainthm}.
\end{proof}

We further remark that the uniform convergence result is a global stability result for a single metastable transition profile $\phi(x,y)$. For a slow bulk diffusion coupled with fast boundary reaction system, one can also use this global stability result for single profile to study the pattern formation of $N$ multilayer transition profile at a finite time.  This will rely on   obtaining a localized version for the uniform stability result in Theorem \ref{mainthm}, which also bases on the fat tail estimate for the transition profile $\phi$. We will leave the multilayer pattern formation and its slow motion persistence as a future study.

\appendix

\section{Vectorial dislocation model with an interfacial misfit energy}\label{app1}
We briefly introduce the physical model and the associated total energy for dislocations. Then we derive the 2D bulk-interface interactions through gradient flows of a simplified total energy $E$ in \eqref{totalE}. The associated energy dissipation law \eqref{energyD} is important for the proof of the uniform convergence of the dynamic solution.

A dislocation core is a  microscopic region of heavily
distorted atomistic structures with shear displacement jump across a slip plane $\Gamma:=\{(x,y,z);~y=0\}$.  The propagation of a dislocation core, i.e., distorted displacement profile, will eventually lead to plastic deformation with low energy barrier. 
Unlike the classical dislocation theory \cite{HL}, which  regards the dislocation core as a
singular point and use  linear elasticity theory,  the Peierls-Nabarro  model introduced by \textsc{Peierls and Nabarro} \cite{Peierls, Nabarro}  is
a multiscale continuum model for displacement $\mathbf{u}=(u_1, u_2, u_3)$ that incorporates the atomistic effect by introducing a nonlinear interfacial
potential $W$ on the slip plane $\Gamma$.  Two elastic continua $y>0$ and $y<0$ are connected by the nonlinear atomistic potential
$W([u_1])$ depending on shear displacement jump $[u_1]$ across the slip plane $\Gamma$; see Fig. \ref{fig1} (Left). 
To minimize the elastic energy and the misfit energy induced by dislocation, the steady solution to PN model is a minimization problem
\begin{equation}\label{mini}
\mathbf{u}=\text{argmin}~\left\{ E_{\text{elastic}}(\mathbf{u})+ E_{\text{misfit}}(\mathbf{u})\right\}= \text{argmin}~ \left\{\frac12\int_{\mathbb{R}^3\backslash \Gamma} \sigma:\eps \ud x \ud y \ud z  + \int_\Gamma W([u_1]) \ud x \ud z \right\},
\end{equation}
among all displacements fields $\mathbf u$ with bi-states far field condition $[u_1](\pm\8)=\pm2$, where   $\eps$ is the strain tensor and $\sigma$ is the stress tensor. Here  $W$ is a Ginzburg-Landau type potential on interface which  determines the stable  states for the shear displacement and drives metastable pattern formation; see \eqref{potential}. Without loss of generality,  we assume a symmetric displacement in the upper/lower elastic bulks and fix the total magnitude of the dislocation, so the minimization constraint   in \eqref{mini} can be simplified as that for any $z$,
$
\lim_{x\to \pm\8}u_1(x, 0,z) = \pm 1.
$
A simplified total energy (see \eqref{totalE}) in terms of the scalar shear displacement function $u=u_1(x,y,z)$ is commonly used in  mathematical analysis
\begin{equation} 
E( {u})=\frac 1 2\int_{\bR^2_+} |\nabla  {u}|^2\ud x\ud y+\int_\Gamma W( {u})\ud x.
\end{equation}
For a straight dislocation with uniform displacement in $z$ direction, the equivalence between the minimizers of the simplified energy $E$ and of the original physical energy in \eqref{mini} are proved in \cite{gao2021revisit}. Meanwhile, if the misfit potential only depends on the shear jump displacement across the slip plane $\Gamma$, then the rigidity result in \cite{dong2021existence, gao2020existence} shows the steady solution must be a straight dislocation with 1D profile. Therefore, we will  use the simplified total energy \eqref{totalE} to derive and study the full dynamics of dislocations expressed in terms of scalar displacement $u(x,y)$. However, the full dynamics and global stability for the true vectorial dislocation model is a challenging future project.

\subsection{Model derivation via gradient flow} 
In this section, we derive the bulk-interface interactive dynamics \eqref{maineq}  via gradient flow of the total energy \eqref{totalE}.
Consider the upper half plane 
\begin{equation}
\bR^2_+:=\{(x,y)\in \mathbb{R}^2; ~ y> 0\}; \quad \overline{\bR^2}_+:=\{(x,y)\in \mathbb{R}^2; ~ y\geq 0\}.
\end{equation}
Denote 
$\Gamma:=\{(x,y); ~y=0\}, \,\, u|_{\Gamma}=u(x,0).$

The full dynamics of dislocation motion is essentially determined by the total energy and how the energy change against frictions for the bulks and on the slip plane. First, one can compute the rate of change of total energy w.r.t any virtual velocity $\dot{u}$. Then to determine the true velocity by Onsager's linear response theory  \cite{Onsager}, we use the simplest quadratic Rayleigh dissipation functional  including frictions in the bulks and on slip plane $\Gamma$ as the dissipation metric.

For any  velocities $\dot{u}, \dot{v}$, choose the quadratic  Rayleigh dissipation functional 
$$g(\dot{u}, \dot{v}):=\int_{\mathbb{R}^2_+} \dot{u} \dot{v} \ud x \ud y + \int_{\Gamma} \dot{ u}_{\Gamma} \dot{ v}_{\Gamma} \ud x. $$
Here, without loss of generality, we take same friction coefficients for the bulk velocity and for its trace on slip plane $\Gamma$.
Then the gradient flow of $E$ with respect to metric $g(\cdot, \cdot)$ is $g(\pt_t u , \dot{u}) = - \frac{\ud}{\ud \eps}\big|_{\eps=0} E(u+\eps \dot{u})$ for any any virtual velocity $\dot{u}$.  After calculating the first variation of $E$, this gradient flow  reads 
\begin{equation}\label{gf}
\begin{aligned}
g(\pt_t u , \dot{u}) =& - \frac{\ud}{\ud \eps}\big|_{\eps=0} E(u+\eps \dot{u}) \\
=& \int_{\mathbb{R}^2_+} \nabla u \nabla \dot{u} \ud x \ud y + \int_\Gamma W'(u) \dot{u} \ud x \\
=& -\int_{\mathbb{R}^2_+} \Delta u \dot{u} \ud x \ud y + \int_{\Gamma}[\pt_n u + W'(u)] \dot{u} \ud x.
\end{aligned}
\end{equation}
Then by taking arbitrary virtual velocity $\dot{u}$, we conclude the governing equation \eqref{maineq}. From same calculations as \eqref{gf}, we also have the energy dissipation law
\begin{equation}\label{energyD}
\begin{aligned}
\frac{\ud}{\ud t} E(u) =& \int_{\bR^2_+} \nabla u \nabla u_t \ud x \ud y + \int_\Gamma W'(u) u_t \ud x\\
=& -\int_{\bR^2_+} u_t^2 \ud x \ud y - \int_\Gamma u_t^2 \ud x \\
=& -\int_{\bR^2_+} |\Delta u|^2 \ud x \ud y - \int_\Gamma [\pt_y u -W'(u)]^2 \ud x =: -Q(t)\leq 0.
\end{aligned}
\end{equation}

\bibliographystyle{alpha}
\bibliography{dbc}

\end{document}